\title{\vspace{-0.7cm} Ramsey-type problem for an almost monochromatic $K_4$}
\author{
Jacob Fox\thanks{Department of Mathematics, Princeton, Princeton,
NJ. Email: {\tt jacobfox@math.princeton.edu}. Research supported by
an NSF Graduate Research Fellowship and a Princeton Centennial
Fellowship.} \and Benny Sudakov\thanks{Department of Mathematics,
UCLA,  Los Angeles, CA 90095 and Institute for Advanced Study,
Princeton, NJ. Email: {\tt bsudakov@math.ucla.edu}. Research
supported in part by NSF CAREER award DMS-0546523, NSF grants
DMS-0355497 and DMS-0635607, by a USA-Israeli BSF grant, and by the
State of New Jersey.}}

\documentclass[11pt]{article}
\usepackage{amsfonts,amssymb,amsmath,latexsym}

\newenvironment{proof}
      {\medskip\noindent{\bf Proof.}\hspace{1mm}}
      {\hfill$\Box$\medskip}

\def\qed{\ifvmode\mbox{ }\else\unskip\fi\hskip 1em plus 10fill$\Box$}

\newtheorem{theorem}{Theorem}[section]

\newtheorem{lemma}[theorem]{Lemma}
\newtheorem{proposition}[theorem]{Proposition}
\newtheorem{corollary}[theorem]{Corollary}

\begin{document}
\date{}

\maketitle

\begin{abstract}
In this short note we prove that there is a constant $c$ such that
 every $k$-edge-coloring of the complete graph $K_n$ with $n \geq 2^{ck}$ contains
a $K_4$ whose edges receive at most two colors. This improves on a
result of Kostochka and Mubayi, and is the first exponential bound
for this problem.
\end{abstract}

\section{Introduction}
The {\it Ramsey number} $R(t;k)$ is the least positive integer $n$
such that every $k$-coloring of the edges of the complete graph
$K_n$ contains a monochromatic $K_t$. Schur in 1916 showed that
$R(3;k)$ is at least exponential in $k$ and at most a constant times
$k!$. Despite various efforts over the past century to determine the
asymptotics of $R(t;k)$, there were only improvements in the
exponential constant in the lower bound and the constant factor in
the upper bound. It is a major open problem to determine whether or
not there is a constant $c$ such that $R(3;k) \leq c^k$ for all $k$
(see, e.g., the monograph \cite{GrRoSp}).

In 1981, Erd\H{o}s \cite{Er} proposed to study the following
generalization of the classical Ramsey problem. Let $p,q$ be
positive integers with $2 \leq q \leq {p \choose 2}$. A
$(p,q)$-coloring of $K_n$ is an edge-coloring such that every copy
of $K_p$ receives at least $q$ distinct colors. Let $f(n,p,q)$ be
the minimum number of colors in a $(p,q)$-coloring of $K_n$.
Determining the numbers $f(n,p,2)$ is equivalent to determining the
multicolor Ramsey numbers $R(p;k)$ as an edge-coloring is a
$(p,2)$-coloring if and only if it does not contain a monochromatic
$K_p$. Over the last two decades, the study of $f(n,p,q)$ drew a lot
of attention. Erd\H{o}s and Gy\'arf\'as \cite{ErGy} proved several
results on $f(n,p,q)$, e.g., they determined for which fixed $p$ and
$q$ we have $f(n,p,q)$ is at least linear in $n$, quadratic in $n$,
or ${n \choose 2}$ minus a constant. For fixed $p$, they also gave
bounds on the smallest $q$ for which $f(n,p,q)$ is asymptotically
${n \choose 2}$. These bounds were significantly tightened by
S\'ark\"ozy and Selkow \cite{SaSe1} using Szemer\'edi's Regularity
Lemma. In a different paper, S\'ark\"ozy and Selkow \cite{SaSe} show
that $f(n,p,q)$ is linear in $n$ for at most $\log p$ values of $q$.
(Here, and throughout the paper, all logarithms are base $2$.) There
are also results on the behavior of $f(n,p,q)$ for particular values
of $p$ and $q$. Mubayi \cite{Mu1} gave an explicit construction of
an edge-coloring which together with the already known lower bound
shows that $f(n,4,4) = n^{1/2+o(1)}$. Using Behrend's construction
of a dense set with no arithmetic progressions of length three,
Axenovich \cite{Ax} showed that $\frac{1+\sqrt{5}}{2}n-3 \leq
f(n,5,9) \leq 2n^{1+c/\sqrt{\log n}}$. These examples demonstrate
that special cases of $f(n,p,q)$ lead to many interesting problems.

As was pointed out by Erd\H{o}s and Gy\'arf\'as \cite{ErGy}, one of
the most intriguing problems among the small cases is the behavior
of $f(n,4,3)$. This problem can be rephrased in terms of another
more convenient function. Let $g(k)$ be the largest positive integer
$n$ for which there is a $k$-edge-coloring of $K_n$ in which every
$K_4$ receives at least three colors, i.e., for which $f(n,4,3) \leq
k$. Restated, $g(k)+1$ is the smallest positive integer $n$ for
which every $k$-edge-coloring of the edges of $K_n$ contains a $K_4$
that receives at most two colors. In 1981, by an easy application of
the probabilistic method, Erd\H{o}s \cite{Er} showed that $g(k)$ is
superlinear in $k$. Later, Erd\H{o}s and Gy\'arf\'as used the
Lov\'asz Local Lemma to show that $g(k)$ is at least quadratic in
$k$. Mubayi \cite{Mu} improved these bounds substantially, showing
that $g(k) \geq 2^{c(\log k)^2}$ for some absolute positive constant
$c$. On the other hand, the progress on the upper bound was much
slower. Until very recently, the best result was of the form
$g(k)<k^{ck}$ for some constant $c$, which follows trivially from
the multicolor $k$-color Ramsey number for $K_4$. This bound was
improved by Kostochka and Mubayi \cite{KoMu}, who showed that
$g(k)<(\log k)^{ck}$ for some constant $c$. Here we further extend
their neat approach and obtain the first exponential upper bound for
this problem.

\begin{theorem}\label{main} For $k>2^{100}$, we have $g(k) <2^{2000k}$.
\end{theorem}

While it is a longstanding open problem to determine whether or not
$R(t;k)$ grows faster than exponential in $k$, it is not difficult
to prove an exponential upper bound if we restrict the colorings to
those that do not contain a rainbow $K_s$ for fixed $s$. Let
$M(k,t,s)$ be the minimum $n$ such that every $k$-edge-coloring of
$K_n$ has a monochromatic $K_t$ or a rainbow $K_s$. Axenovich and
Iverson \cite{AxIv} showed that $M(k,t,3) \leq 2^{kt^2}$. We improve
 on their bound by showing that $M(k,t,s)
\leq s^{4kt}$ for all $k,t,s$. In the other direction, we prove that
for all positive integers $k$ and $t$ with $k$ even and $t \geq 3$,
$M(k,t,3) \geq 2^{kt/4}$, thus determining $M(k,t,3)$ up to a
constant factor in the exponent.

The rest of this paper is organized as follows. In the next section,
we prove our main result, Theorem \ref{main}. In Section
\ref{monorainbow}, we study the Ramsey problem for colorings without
rainbow $K_s$. The last section of this note contains some
concluding remarks. Throughout the paper, we systematically omit
floor and ceiling signs whenever they are not crucial for the sake
of clarity of presentation. We also do not make any serious attempt
to optimize absolute constants in our statements and proofs.

\section{Proof of Theorem \ref{main}}

Our proof develops further on ideas in \cite{KoMu}. Like the
Kostochka-Mubayi proof, we show that the $K_4$ we find is
monochromatic or is a $C_4$ in one color and a matching in the other
color. Call a coloring of $K_t$ {\it rainbow} if all ${t \choose 2}$
edges have different colors. Let $g(k,t)$ be the largest positive
integer $n$ such that there is a $k$-edge-coloring of $K_n$ with no
rainbow $K_t$, and in which the edges of every $K_4$ have at least
three colors. We will study $g(k)$ by investigating the behavior of
$g(k,t)$.

Before jumping into the details of the proof of Theorem \ref{main},
we first outline the proof idea. Note that $g(k)=g(k,k)$ for $k>2$
as a rainbow $K_k$ would use ${k \choose 2}>k$ colors. We give a
recursive upper bound on $g(k,t)$ which implies Theorem \ref{main}.
We first prove a couple of lemmas which show that in any
edge-coloring without a rainbow $K_t$, there are many vertices that
have large degree in some color $i$. We then apply a simple
probabilistic lemma to find a large subset $V_2$ of vertices such
that every vertex subset of size $d$ (with $d\ll t$) has many common
neighbors in color $i$. We use this to get an upper bound on
$g(k,t)$ as follows. Consider a $k$-edge-coloring of $K_n$ with
$n=g(k,t)$ without a rainbow $K_t$ and with every $K_4$ containing
at least three colors. There are two possible cases. If there is no
rainbow $K_d$ in the set $V_2$, then we obtain an upper on $g(k,t)$
using the fact that $|V_2|$ has size at most $g(k,d)$. If there is a
set $R \subset V_2$ of $d$ vertices which forms a rainbow $K_d$,
then the ${d \choose 2}$ colors that appear in this rainbow $K_d$
cannot appear in the edges inside the set $N_i(R)$ of vertices that
are adjacent to every vertex in $R$ in color $i$, for otherwise we
would obtain a $K_4$ having at most two colors (the color $i$ and
the color that appears in both $R$ and in $N_i(R)$). In this case we
obtain an upper bound on $g(k,t)$ using the fact that $|N_i(R)| \leq
g(k-{d \choose 2},t)$. Finally, if the coloring has no rainbow $K_d$
with $d$ constant, it is easy to show an exponential upper bound.

For an edge-coloring of $K_n$, a vertex $x$, and a color $i$, let
$d_i(x)$ denote the degree of vertex $x$ in color $i$. Our first
lemma shows that if, for every vertex $x$ and color $i$, $d_i(x)$ is
not too large, then the coloring contains many rainbow cliques.

\begin{lemma}\label{useful1}
If an edge-coloring of the complete graph $K_n$ satisfies $d_i(x)
\leq \delta n$ for each $x \in V(K_n)$ and each color $i$, then this
coloring has at most $\frac{5}{8}\delta t^{4} {n \choose t}$
non-rainbow copies of $K_t$.
\end{lemma}
\begin{proof}
If a $K_t$ is not rainbow, then it has two adjacent edges of the
same color or two nonadjacent edges of the same color. We will use
this fact to give an upper bound on the number of $K_t$s that are
not rainbow.

Let $\nu(i,t,n)$ be the number of copies of $K_t$ in $K_n$ in which
there is a vertex in at least two edges of color $i$. We can first
choose the vertex, and then the two edges with color $i$. Hence, the
number of $K_t$s for which there is a vertex with degree at least
two in some color is at most \begin{eqnarray*} \sum_{i} \nu(i,t,n) &
\leq & \sum_{i} \sum_{x \in V} {d_i(x) \choose 2}{n-3 \choose t-3}
\leq  n\delta^{-1}{\delta n \choose 2}{n-3 \choose t-3} \\ & \leq &
\frac{\delta n^3}{2}\left(\frac{t}{n}\right)^3{n \choose t} =
\frac{1}{2}\delta t^3{n \choose t},\end{eqnarray*} where we used
that $d_i(x) \leq \delta n$ together with the convexity of the
function $f(y)={y \choose 2}$.

Let $\psi(i,t,n)$ be the number of copies of $K_t$ in $K_n$ in which there
is a matching of size at least two in color $i$. Let $e_i$ denote
the number of edges of color $i$. Since $$e_i \leq
\frac{n}{2}\max_{x \in V} d_i(x) \leq \frac{\delta}{2} n^2,$$ then
the number of $K_t$s in which there is a matching of size
at least two in some color is at most
$$\sum_{i} \psi(i,t,n) \leq \sum_{i} {e_i \choose 2}{n-4 \choose
t-4} \leq \delta^{-1}{\delta n^2/2 \choose 2}{n-4 \choose t-4} \leq
\frac{\delta t^4}{8}{n \choose t},$$ where again we used the
convexity of the function $f(y)={y \choose 2}$. Hence, the number of
$K_t$s which are not rainbow is at most $\frac{1}{2}\delta t^3{n
\choose t}+\frac{1}{8}\delta t^4{n \choose t} \leq \frac{5}{8}\delta
t^4{n \choose t}$, completing the proof.
\end{proof}

For the proof of Theorem \ref{main}, we do not need the full
strength of this lemma since we will only use the existence of at
least one rainbow $K_t$. We also would like to mention the following
stronger result. Call an edge-coloring {\it $m$-good} if each color
appears at most $m$ times at each vertex. Let $h(m,t)$ denote the
minimum $n$ such that every $m$-good edge-coloring of $K_n$ contains
a rainbow $K_t$. The above lemma demonstrates that $h(m,t)$ is at
most $mt^4$. It is shown by Alon, Jiang, Miller, and Pritikin
\cite{AlJiMiPr} that there are constant positive constants $c_1$ and
$c_2$ such that
$$c_1mt^3/\log t \leq h(m,t) \leq c_2mt^3/\log t.$$

The following easy corollary of Lemma \ref{useful1} demonstrates
that in every $k$-edge-coloring without a rainbow $K_t$, there is a
color and a large set of vertices which have large degree in that
color.

\begin{corollary}\label{veryuseful}
In every $k$-edge-coloring of $K_n$ without a rainbow $K_t$, there
is a subset $V_1 \subset V(K_n)$ with $|V_1| \geq \frac{n}{2k}$ and
a color $i$ such that $d_i(x) \geq \frac{n}{2t^4}$ for each vertex
$x \in V_1$.
\end{corollary}
\begin{proof}
Let $V' \subset V(K_n)$ be those vertices $x$ for which there is a
color $i$ such that $d_i(x) \geq \frac{n}{2t^4}$.

Case 1: $|V'| <n/2$. In this case, letting $V''=V(K_n) \setminus
V'$, $|V''| \geq n/2$ and no vertex in $V''$ has degree at least
$\frac{n}{2t^4} \leq |V''|/t^4$ in any given color. By Lemma
\ref{useful1} applied to the coloring of $K_n$ restricted to $V''$
with $\delta=t^{-4}$, there are at least $\frac{3}{8}{|V''| \choose
t}$ rainbow $K_t$s, contradicting the assumption that the coloring
is free of rainbow $K_t$s.

Case 2: $|V'| \geq n/2$. In this case, by the pigeonhole principle,
there is a color $i$ and at least $\frac{n}{2k}$ vertices $x$ for
which $d_i(x) \geq \frac{n}{2t^4}$, completing the proof.
\end{proof}

The following lemma is essentially the same as results in
\cite{KoRo} and \cite{Su}. Its proof uses a probabilistic argument
commonly referred to as dependent random choice, which appears to be
a powerful tool in proving various results in Ramsey theory (see,
e.g., \cite{FoSu} and its references). In a graph $G$, the {\it
neighborhood} $N(v)$ of a vertex $v$ is the set of vertices adjacent
to $v$. For a vertex subset $U$ of a graph $G$, the {\it common
neighborhood} $N(U)$ is the set of vertices adjacent to all vertices
in $U$.

\begin{lemma}\label{drc}
Let $G=(V,E)$ be a graph with $n$ vertices and $V_1 \subset V$ be a
subset with $|V_1|=m$ in which each vertex has degree at least
$\alpha n$. If $\beta\leq m^{-d/h}$, then there is a subset $V_2
\subset V_1$ with $|V_2| \geq \alpha^h m -1$ such that every
$d$-tuple in $V_2$ has at least $\beta n$ common neighbors.
\end{lemma}
\begin{proof}
Let $U=\{x_1,\ldots,x_h\}$ be a subset of $h$ random vertices from
$V$ chosen uniformly with repetitions, and let $V_1'=N(U) \cap V_1$.
We have
$$\mathbb{E}[|V_1'|] =\sum_{v \in V_1}{\bf Pr}(v \in
N(U))=\sum_{v \in V_1}\left(\frac{|N(v)|}{n}\right)^h \geq
\alpha^{h}m.$$

The probability that a given set $W\subset V_1$ of vertices is
contained in $V_1'$ is $\left(\frac{|N(W)|}{n}\right)^h$. Let $Z$
denote the number of $d$-tuples in $V_1'$ with less than $\beta n$
common neighbors. So $$\mathbb{E}[Z]=\sum_{W \subset V_1,
|W|=d,|N(W)|<\beta n} {\bf Pr}(W \subset V_2) \leq {m \choose
d}\beta^h \leq m^d \beta^h \leq 1.$$ Hence, the expectation of
$|V_1'|-Z$ is at least $\alpha^h m-1$ and thus, there is a choice
$U_0$ for $U$ such that the corresponding value of $|V_1'|-Z$ is at
least $\alpha^h m - 1$. For every $d$-tuple $D$ of vertices of
$V_1'$ with less than $\beta n$ common neighbors, delete a vertex
$v_D \in D$ from $V_1'$. Letting $V_2$ be the resulting set, it is
clear that $V_2$ has the desired properties, completing the proof.
\end{proof}

The proof of the next lemma uses the standard pigeonhole argument
together with Lemma \ref{useful1}.

\begin{lemma}\label{lastpiece}
Let $d,k$ be integers with $d,k \geq 2$. Then every
$k$-edge-coloring of $K_n$ with $n \geq d^{12k}$ and without a
rainbow $K_d$ has a monochromatic $K_4$. In particular, we have
$g(k,d) < d^{12k}$.
\end{lemma}
\begin{proof}
Suppose for contradiction that there is a $k$-edge-coloring of $K_n$
with $n \geq d^{12k}$ and without a rainbow $K_d$ and without a
monochromatic $K_4$. By Lemma \ref{useful1} with $t=d$ and
$\delta=d^{-4}$, this graph contains a vertex $x_1$ with degree at
least $\frac{n}{d^4}$ in some color $c_1$. Pick this vertex $x_1$
out and let $N_1$ be the set of vertices adjacent to $x_1$ by color
$c_1$. We will define a sequence $x_1,\ldots,x_{2k+1}$ of vertices,
a sequence $c_1,\ldots,c_{2k+1}$ of colors, and a sequence $V(K_n)
\supset N_1 \supset \ldots \supset N_{2k+1}$ of vertex subsets. Once
$x_j$, $c_j$, and $N_j$ have been defined, pick a vertex $x_{j+1}$
in $N_{j}$ in at least $\frac{|N_j|}{d^4}$ edges in some color
$c_{j+1}$ with other vertices in $N_j$. Pick this vertex $x_{j+1}$
out and let $N_{j+1}$ be the set of vertices in $N_j$ that are
adjacent to $x_j$ by color $c_j$. Note that $|N_{j+1}| \geq
d^{-4}|N_j|$ so
$$|N_{2k+1}| \geq (d^{-4})^{2k+1}n \geq 1.$$ Therefore, there is a
color $c$ such that $c$ is represented at least three times in the
list $c_1,\ldots,c_{2k+1}$ and the three vertices
$x_{j_1},x_{j_2},x_{j_3}$ together with a vertex from $N_{2k+1}$
form a monochromatic $K_4$ in color $c$, where
$c_{j_1}=c_{j_1}=c_{j_3}=c$ with $j_1<j_2<j_3$.
\end{proof}

\begin{lemma}\label{mainhelp} Let $d,k,t$ be positive integers
with $3 \leq d \leq t$ and $d \geq 40\log t$. If $k \geq {d \choose
2}$, then
\begin{equation}\label{eq} g(k,t) \leq
\max\left(4kg(k,t)^{\frac{20\log t}{d}}g(k,d),2^{{d \choose
2}}g\Big(k-\left(d \atop 2\right),t\Big)\right).\end{equation}
Otherwise, we have $g(k,t)=g(k,d)$.
\end{lemma}
\begin{proof}
Note that if $k<{d \choose 2}$, then a $k$-edge-coloring cannot have
a rainbow $K_d$. Therefore, $g(k,t)=g(k,d)$ in this case. So we
assume $k \geq {d \choose 2}$. By the definition of $g(k,t)$, there
is a $k$-edge-coloring of $K_n$ with $n=g(k,t)$ with no rainbow
$K_t$ and in which every $K_4$ receives at least three colors.
Consider such a coloring. By Corollary \ref{veryuseful}, there is a
color $i$ and a subset $V_1 \subset V(K_n)$ with $|V_1| \geq
\frac{n}{2k}$ and $d_i(x) \geq \frac{n}{2t^4}$ for every vertex $x
\in V_1$. Apply Lemma \ref{drc} to the graph of color $i$ with
$\alpha=\frac{1}{2t^4}$, $\beta=2^{-{d \choose 2}}$, $m=|V_1| \geq
\frac{n}{2k}$, and $h=4d^{-1}\log n$. We can apply Lemma \ref{drc}
since $\beta<2^{-d^2/4}=n^{-d/h} \leq |V_1|^{-d/h}$. So there is a
subset $V_2 \subset V_1$ such that
$$|V_2| \geq \alpha^h m -1 \geq \alpha^h m/2 \geq (2t^4)^{-4d^{-1}\log
n} \cdot \frac{n}{4k} \geq n^{1-\frac{20\log t}{d}}/(4k)$$ and every
subset of $V_2$ of size $d$ has common neighborhood at least $\beta
n=2^{-{d \choose 2}}n$ in color $i$.

There are two possibilities: either every $K_d$ in $V_2$ is not
rainbow, or there is a $K_d$ in $V_2$ that is rainbow. In the first
case, the $k$-edge-coloring restricted to $V_2$ is free of rainbow
$K_d$, so $$g(k,d) \geq |V_2| \geq n^{1-\frac{20\log t}{d}}/(4k).$$
Since $n=g(k,t)$, we can restate this inequality as
$$g(k,t) \leq 4kg(k,t)^{\frac{20\log t}{d}}g(k,d).$$ In the second case, there is a rainbow $d$-tuple $R \subset V_2$
such that $N_i(R)$, the common neighborhood of $R$ in color $i$, has
cardinality at least $\beta n$. The ${d \choose 2}$ colors present
in $R$ can not be present in $N_i(R)$ since otherwise we would have
a $K_4$ using only two colors (the color $i$ and the color that
appears in both $R$ and in $N_i(R)$). In this case we have
$$g\Big(k-\left(d \atop 2\right),t\Big) \geq |N_i(R)| \geq \beta
n=2^{-{d \choose 2}}g(k,t).$$ In either case we have
$$g(k,t) \leq \max\left(4kg(k,t)^{\frac{20\log t}{d}}g(k,d),2^{{d \choose
2}}g\Big(k-\left(d \atop 2\right),t\Big)\right),$$ which completes
the proof.
\end{proof}

Having finished all the necessary preparation, we are now ready to
prove Theorem \ref{main}, which says that $g(k) \leq 2^{2000k}$ for
$k>2^{100}$. The iterated logarithm $\log^* n$ is defined by $\log^*
n =0$ if $n \leq 1$ and otherwise $\log^* n =1+\log^* \log n$. It is
straightforward to verify that $\log^* n < \log n$ holds for $n>8$.

\vspace{0.2cm}
\noindent
{\bf Proof of Theorem \ref{main}:} Note that $g(k)=g(k,k)$ since no
$k$-edge-coloring contains a rainbow $K_k$. Assume $k>2^{100}$ and
suppose for contradiction that there is a $k$-edge-coloring of $K_n$
with $n=g(k) \geq 2^{2000k}$ such that every $K_4$ has at least
three colors.

Let $t_1=k$, and if $t_i>2^{100}$, let $t_{i+1}=(\log t_i)^2$. We
first exhibit several inequalities which we will use. We have
$t_{i+1} > 100 \log t_i$ and $20\frac{\log t_i}{t_{i+1}} = 20/\log
t_{i} \leq \frac{1}{5}$. Let $\ell$ be the largest positive integer
for which $t_{\ell}$ is defined, so $100<t_{\ell} \leq 2^{100}$.
Note that $\ell<2\log^* k$ as one can easily check that
$t_{j+1}=(\log t_j)^2 =(2 \log \log t_{j-1})^2< \log t_{j-1}$. Since
$\ell<2\log^*k \leq 2\log k$ and $n \geq 2^{2000k}$, then $(4k)^{\ell}<n^{1/12}$. We have
$\sum_{i=1}^{\ell-1} 20/\log t_{i}<1/4$ since the largest term in
the sum is $20/\log t_{\ell-1}<1/5$, and $20/\log
t_{\ell-i}<2^{-5i}$ for $2 \leq i \leq \ell-1$. Putting this
together, we have
$$(4k)^{\ell-1}n^{\sum_{i=1}^{\ell-1}20/\log t_{i}}<n^{1/3}.$$

To get an upper bound on $g(k,k)$ we repeatedly apply Lemma
\ref{mainhelp}. Given $k'$ and $t=t_i$, to bound $g(k',t)$, we use
this lemma with $d=t_{i+1}$. Note that we have $d=t_{i+1}>100\log
t_i$, so indeed the condition of the lemma holds. If $k'<{t_{i+1}
\choose 2}$, then $g(k',t_i)=g(k',t_{i+1})$. Otherwise, we
have one of two possible upper bounds given by (\ref{eq}). If the
maximum of the two terms in (\ref{eq}) is the left bound, then
$$g(k',t) \leq 4k'g(k',t)^{\frac{20\log t}{d}}g(k',d) \leq
4kn^{\frac{20\log t}{d}}g(k',d)=4kn^{20/\log t_i}g(k',d),$$
otherwise we have $g(k',t) \leq 2^{j}g(k'-j,t)$ with $j={d \choose
2}$. Since $\frac{g(k',t)}{g(k',d)} \leq 4kn^{20/\log t_i}$, we can
only accumulate up to a total upper bound factor of
$$\prod_{i=1}^{l-1} 4kn^{20/\log
t_i}=(4k)^{\ell-1}n^{\sum_{i=1}^{\ell-1}20/\log t_{i}}<n^{1/3}$$ in
all of the applications of the left bound. When we use the right
bound, we pick up a factor of $\frac{g(k',t)}{g(k'-j,t)} \leq 2^j$
with $j={d \choose 2}$ and also decrease $k'$ by $j$. So this can only give another multiplicative
factor of at most $2^k$ in all of the applications of the right
bound.

Notice that when we finish repeatedly applying Lemma \ref{mainhelp}
we end up with a term  of the form
$g(k_0,t_{\ell})$ with $k_0 \leq k$. In that case, we use that $t_{\ell} \leq 2^{100}$ together with
Lemma \ref{lastpiece} to bound it by $g(k,t_{\ell}) \leq
t_{\ell}^{12k} \leq 2^{1200k}$. Putting this all together, we obtain
the upper bound
$$n = g(k)=g(k,k)< n^{1/3}2^kg(k,t_{\ell})<
2^{1201k}n^{1/3},$$ which implies that $n<2^{2000k}$. This completes the proof.
\qed

\section{Monochromatic or Rainbow Cliques}
\label{monorainbow}

In this section, we prove bounds on the smallest $n$, denoted by
$M(k,t,s)$, such that every $k$-edge-coloring of $K_n$ contains a
monochromatic $K_t$ or a rainbow $K_s$. The following proposition is
a straightforward generalization of Lemma \ref{lastpiece}.

\begin{proposition}\label{fun}
We have $M(k,t,s) \leq s^{4kt}$.
\end{proposition}

Let $M_s(t_1,\ldots,t_k)$ be the maximum $n$ such that there is a
$k$-edge-coloring of $K_n$ with colors $\{1,\ldots,k\}$ without a
rainbow $K_s$ and without a monochromatic $K_{t_i}$ in color $i$ for
$1 \leq i \leq k$. The above proposition follows from repeated
application of the following recursive bound.

\begin{lemma} We have
$$M_s(t_1,\ldots,t_k) \leq s^4\max_{1 \leq i \leq k}
M_s(t_1,\ldots,t_i-1,\ldots,t_k).$$
\end{lemma}
\begin{proof}
By Lemma \ref{useful1}, for every edge-coloring of $K_n$ without a
rainbow $K_s$, there is a vertex $v$ with degree at least $n/s^4$ in
some color $i$. If the coloring of $K_n$ does not contain a
monochromatic $K_{t_i}$ in color $i$, then the neighborhood of $v$
in color $i$ has at least $n/s^4$ vertices and does not contain
$K_{t_i-1}$ in color $i$, completing the proof.
\end{proof}

Using a slightly better estimate by Alon et al. \cite{AlJiMiPr}
(which we mentioned earlier) instead of Lemma \ref{useful1}, one can
improve the constant in the exponent of the above proposition from
$4$ to $3$. Together with the next lemma, Proposition \ref{fun}
determines $M(k,t,3)$ up to a constant factor in the exponent.

\begin{lemma} \label{prop1}
For all positive integers $k$ and $t$ with $k$ even and $t \geq 3$,
we have $M(k,t,3)> 2^{kt/4}$.
\end{lemma}
\begin{proof}
To prove the lemma, it suffices by induction to prove $M(k,t,3)-1
\geq 2^{t/2}\left(M(k-2,t,3)-1\right)$ for all $k \geq 2$ and $t
\geq 3$. Consider a $2$-edge-coloring $C_1$ of $K_m$ with
$m=2^{t/2}$ and without a monochromatic $K_t$. Such a
$2$-edge-coloring exists by the well-known lower bound of Erd\H{o}s
\cite{Er1} on the $2$-color Ramsey number $R(t;2)$. Consider also a
$(k-2)$-edge-coloring $C_2$ of $K_r$ with $r=M(k-2,t,3)-1$ without a
rainbow triangle and without a monochromatic $K_t$. We use these two
colorings to make a new edge-coloring $C_3$ of $K_{mr}$ with $k$
colors: we first partition the vertices of $K_{mr}$ into $m$ vertex
subsets $V_1,\ldots,V_m$ each of size $r$, and color any edge
$e=(v,w)$ with $v \in V_i,w \in V_j$, and $i \not = j$ by the color
of $(i,j)$ in the $2$-edge-coloring $C_1$ of $K_m$, and color within
each $V_i$ identical to the coloring $C_2$ of $K_r$. First we show
that coloring $C_3$ has no rainbow triangle. Indeed, consider three
vertices of $K_{mr}$. If all three vertices lie in the same vertex
subset $V_i$, then the triangle between them is not rainbow by the
assumption on coloring $C_2$. If exactly two of the three vertices
lie in the same vertex subset, then the two edges from these
vertices to the third vertex will receive the same color. Finally,
if they lie in three different vertex subsets, then the triangle
between them receives only colors from $C_1$ and is not rainbow
since $C_1$ is a $2$-coloring. Similarly, one can see that coloring
$C_3$ has no monochromatic $K_t$, which completes the proof.
\end{proof}

\section{Concluding Remarks}
\label{concluding}

In this paper we proved that there exists a constant $c$ such that
every $k$-edge-coloring of $K_n$ with $n \geq 2^{ck}$ contains a
$K_4$ whose edges receive at most two colors. On the other hand, for
$n \leq 2^{c(\log k)^2}$, Mubayi constructed a $k$-edge-coloring of
$K_n$ in which every $K_4$ receives at least three colors. There is
still a large gap between these results. We believe that the lower
bound is closer to the truth and the correct growth is likely to be
subexponential in $k$.

Our upper bound is equivalent to $f(n,4,3) \geq (\log n)/4000$ for
$n$ sufficiently large. Kostochka and Mubayi showed that
$f(n,2a,a+1) \geq c_a\frac{\log n}{\log \log \log n}$, where $c_a$
is a positive constant for each integer $a \geq 2$. Like the
Kostochka-Mubayi proof, our proof can be generalized to demonstrate
that for every integer $a \geq 2$ there is $c_a>0$ such that
$f(n,2a,a+1) \geq c_a\log n$ for every positive integer $n$. For
brevity, we do not include the details.

We do not yet have a good understanding of how $M(k,t,s)$, which is
the smallest positive integer $n$ such that every $k$-edge-coloring
of $K_n$ has a monochromatic $K_t$ or a rainbow $K_s$, depends on
$s$. From the definition, it is an increasing function in $s$. For
constant $s$, we showed that $M(k,t,s)$ grows only exponentially in
$k$. On the other hand, for ${s \choose 2}>k$, we have
$M(k,t,s)=R(t;k)$, so understanding the behavior of $M(k,t,s)$ for
large $s$ is equivalent to understanding the classical Ramsey
numbers $R(t;k)$.

\end{document}